\documentclass[12pt,english,a4paper,oneside]{amsart}

\usepackage[cp1251]{inputenc}
\usepackage[english]{babel}
\usepackage{amsmath,amsthm,amssymb}

\usepackage[
a4paper,
left=22mm,
right=18mm,
top=20mm,
bottom=27mm
]{geometry}

\usepackage{xcolor}
\usepackage[colorlinks]{hyperref}

\newtheorem{theorem}{Theorem}[section]
\newtheorem{lemma}[theorem]{Lemma}
\newtheorem{corollary}[theorem]{Corollary}

\numberwithin{equation}{section}
\theoremstyle{definition}

\newtheorem{remark}[theorem]{Remark}
\newtheorem{example}[theorem]{Example}

\begin{document}

\title{Dimension-free Markov--Bernstein-type inequality for $\mathbb R^2$-valued polynomials}

\author{Egor Kosov}

\address{\noindent Egor Kosov,
Centre de Recerca Matem\`atica, Campus de Bellaterra, Edifici~C 08193
Bellaterra (Barcelona), Spain.}
\email{kosoved09@gmail.com}



\begin{abstract}
We prove a dimension-free Markov--Bernstein-type inequality for polynomial
mappings $(f_1,f_2)\colon\mathbb R^n\to\mathbb R^2$ of degree at most $d$. 
For product probability measures whose one-dimensional marginals have
densities bounded by $1$, we obtain a lower bound for the integral of
the Gram determinant of $\nabla f_1$ and $\nabla f_2$ in terms of the
Gram determinant of the $d$-th derivative tensors $D^df_1$ and $D^df_2$,
with a constant depending only on $d$.
For pairs in a fixed Wiener chaos, the resulting Gaussian inequality
answers a question of Nourdin, Nualart, and Poly concerning a lower bound
for the expectation of the Malliavin determinant in terms of the
determinant of the covariance matrix.

\end{abstract}

\maketitle

\section{Introduction}

Let $\mathcal{P}_d(\mathbb{R}^n)$ denote the space of all algebraic polynomials of degree at most $d$
on $\mathbb{R}^n$.
Let $G_H(h_1, h_2)$ denote the Gram matrix of
vectors $h_1$ and $h_2$ from a Euclidean space $H$.
In this paper we prove the following dimension-free inequality for $\mathbb{R}^2$-valued polynomial mappings.

\begin{theorem}\label{th-main}
Let $d\in\mathbb N$. There exists a constant $C(d)>0$ such that, for every
$n\in\mathbb N$, $n\ge2$, every product measure
$\mu=\mu_1\otimes\ldots\otimes\mu_n$, where each $\mu_j$ is a probability measure on $\mathbb R$ with a density bounded by $1$, 
and every pair of polynomials
$f_1, f_2\in \mathcal{P}_d(\mathbb{R}^n)$, one has
\[
\int_{\mathbb R^n}
\det G_{\mathbb R^n}(\nabla f_1,\nabla f_2)\,d\mu
\ge
C(d)
\det
G_{(\mathbb R^n)^{\otimes d}}
(D^df_1,D^df_2),
\]
where $D^df_j$ denotes the $d$-th derivative tensor of $f_j$.
One may take $C(d)=2^{-71d}(d!)^{-4}$.
\end{theorem}

The classical integral Markov--Bernstein inequality~\cite{HST}
controls the $L^p$-norm of a derivative of a polynomial $f$ of degree at most $d$
on $[0,1]$ in terms of the norm of the polynomial itself:
\[
\|f^{(r)}\|_{L^p([0,1])}
\le
C(p, d,r)\|f\|_{L^p([0,1])}.
\]
Dimension-free generalizations of this inequality were studied in
\cite{EI20} for the standard Gaussian measure and in \cite{Kos26-MB}
for $L^p$-norms with $p$ an even integer and product measures
$\mu=\mu_1\otimes\ldots\otimes\mu_n$ whose components have unimodal densities.
In particular, for such measures with densities bounded by $1$, \cite[Theorem~1.6]{Kos26-MB} gives
\begin{equation}\label{eq-MB-1}
\int_{\mathbb{R}^n}|f|^{2}\,d\mu
\ge
C(d)\int_{\mathbb{R}^n}|\nabla f|^{2}\,d\mu
\end{equation}
for every polynomial $f\in\mathcal{P}_d(\mathbb{R}^n)$.
Iterating this inequality gives
\begin{equation}\label{eq-MB-2}
	\int_{\mathbb{R}^n}|f|^{2}\,d\mu
	\ge
	C(d)\|D^df\|_{(\mathbb R^n)^{\otimes d}}^2.
\end{equation}
When the components of $\mu$ are isotropic log-concave measures, this estimate
was obtained in~\cite{GM22}.
Applying \eqref{eq-MB-2} to the partial derivatives of $f$, we obtain an
inequality comparing the norms of the first- and highest-order derivatives:
\[
\int_{\mathbb{R}^n}|\nabla f|^{2}\,d\mu
\ge
C(d)\|D^df\|_{(\mathbb R^n)^{\otimes d}}^2.
\]
Theorem~\ref{th-main} gives a corresponding estimate for
$\mathbb R^2$-valued polynomial mappings. Indeed, since
$G_{\mathbb R^n}(\nabla f_1,\nabla f_2)=J_fJ_f^*$,
where $J_f$ is the Jacobian matrix of $f=(f_1,f_2)$,
Theorem~\ref{th-main} shows that linear independence of the
$d$-th derivative tensors of $f_1$ and $f_2$ forces the
nondegeneracy of $J_f$ and gives quantitative control of this implication.
Such nondegeneracy estimates play an important role in the study of
small-ball probabilities and the regularity of densities of polynomial mappings,
see, for instance,~\cite{KosZh}.

\begin{remark}
Analogues of Theorem~\ref{th-main} for polynomial mappings with more
than two components or for non-product measures fail in general,
see Examples~\ref{ex-three-components} and~\ref{ex-nonproduct} below.
\end{remark}

\subsection{The Gaussian case}

Let $\gamma$ be a Gaussian measure on an abstract Wiener space $E$, and let
$\mathcal H_d(\gamma)$ denote the Wiener chaos of order $d$.
Nourdin, Nualart, and Poly~\cite{NNP} proved that the law of
$f=(f_1,f_2)$, with $f_1,f_2\in\mathcal H_d(\gamma)$, is absolutely
continuous if and only if the expectation of the determinant $\Delta_f$
of its Malliavin matrix is positive.
In connection with this result,
Nourdin, Nualart, and Poly \cite[Question~6.2]{NNP-arxiv} raised the following
natural question concerning a lower bound for the expectation of $\Delta_f$ in terms of the determinant of the covariance matrix
$\operatorname{Cov}_{\gamma}(f)$:

\smallskip

\noindent
{\it
Is it true that there exists a constant $C(d)>0$, depending only on $d$, such
that, for every $f=(f_1,f_2)$ with
$f_1,f_2\in\mathcal H_d(\gamma)$, one has
\begin{equation}\label{NNP-conj}
	\int_E\Delta_f\,d\gamma
	\ge
	C(d)\det\operatorname{Cov}_{\gamma}(f)\quad?
\end{equation}
}

We note that the question was originally formulated for general
$k$-dimensional random vectors with components in a fixed Wiener chaos, but
simple examples show that the analogous estimate fails in general for $k>2$
(for instance, see Example~\ref{ex-three-components} below).

For $d=2$, this question was answered affirmatively in \cite{NNP}.
Nualart and Tudor \cite{NT17} proved \eqref{NNP-conj} for
$d=3$ and $d=4$. For $d\ge5$, they established a weaker substitute involving
additional nonnegative terms:
\[
\int_E\Delta_f\,d\gamma
+
\sum_{j=2}^{\lfloor (d-1)/2\rfloor}
C_j(d)\int_E\Delta_f^{(j)}\,d\gamma
\ge
C(d)\det\operatorname{Cov}_{\gamma}(f),
\]
where $\Delta_f^{(j)}$ denotes the determinant of the $j$th iterated
Malliavin matrix.
In \cite{Kos26}, the expectation of the Malliavin determinant was bounded
below in terms of a weaker quantity given by the supremum of the
squares of the $2\times2$ minors formed from the Hermite coefficient
vectors of $f_1$ and $f_2$.

Theorem~\ref{th-main} implies \eqref{NNP-conj} for every $d\in\mathbb N$
and therefore answers the above question affirmatively. Indeed, by approximation,
it is sufficient to establish \eqref{NNP-conj}
for the standard Gaussian measure $\gamma_n$
on $\mathbb{R}^n$. In this case, $\mathcal H_d(\gamma_n)$
is the orthogonal complement of
$\mathcal P_{d-1}(\mathbb R^n)$ in $\mathcal P_d(\mathbb R^n)$
with respect to the inner product in $L^2(\gamma_n)$, and
\[
\Delta_f
=
\det (J_fJ_f^*)
=
\det G_{\mathbb R^n}(\nabla f_1,\nabla f_2).
\]
Moreover, the Hermite polynomial expansions of $f_1$ and $f_2$
(see \cite[Definition~1.3.1]{Gaus}) imply that
\[
\det
G_{(\mathbb R^n)^{\otimes d}}
(D^df_1,D^df_2)
=
(d!)^2
\det\operatorname{Cov}_{\gamma_n}(f)
\]
for $f_1,f_2\in\mathcal H_d(\gamma_n)$.
Thus, \eqref{NNP-conj} for $\gamma_n$ follows directly from
Theorem~\ref{th-main}, and we obtain the following corollary.

\begin{corollary}
Let $\gamma$ be a Gaussian measure on an abstract Wiener space $E$,
let $d\in\mathbb N$, and let $f=(f_1,f_2)$ with
$f_1,f_2\in\mathcal H_d(\gamma)$. Then
\[
\int_E\Delta_f\,d\gamma
\ge
2^{-71d}(d!)^{-2}
\det\operatorname{Cov}_{\gamma}(f).
\]
\end{corollary}

\subsection{Ingredients of the proof}

The proof of Theorem~\ref{th-main} has two ingredients.
The first, and more important, ingredient is a dimension-free algebraic
estimate for the coefficients of the partial Jacobian determinants of
an $\mathbb R^2$-valued polynomial mapping. We reduce this problem
to two variables by restricting multivariate polynomials to subspaces
spanned by two vectors whose coordinates lie on the complex unit circle,
and then averaging over the choices of these vectors.
The second ingredient is a scalar estimate, analogous to
\eqref{eq-MB-2}, that bounds the $L^2$-norm of a polynomial from below
in terms of its leading coefficients under product measures with
bounded marginal densities.

To state these estimates precisely, we introduce some notation.
For $\mathbb{K}\in\{\mathbb{R},\mathbb{C}\}$,
let $\mathcal P_d(\mathbb K^n;\mathbb K^k)$
denote the space of all polynomial mappings
\[
f=(f_1,\ldots,f_k)\colon \mathbb K^n\to\mathbb K^k
\] 
of the form
\[
f(z)=\sum_{|\mathbf m|\le d} a_{\bf m}(f)z^{\mathbf m},
\]
where $z=(z_1,\ldots,z_n)\in\mathbb{K}^n$,
$\mathbf m=(m_1,\ldots,m_n)\in\mathbb Z_+^n$,
$|\mathbf m|=m_1+\ldots+m_n$,
$z^{\mathbf m}=z_1^{m_1}\ldots z_n^{m_n}$, and
\[
a_{\bf m}(f)=
\bigl(a_{\bf m}(f_1),\ldots,a_{\bf m}(f_k)\bigr)\in\mathbb K^k.
\]
For nonzero $f$, we set
\[
d(f):=\max\{|\mathbf m|\colon a_{\bf m}(f)\ne0\},
\]
and for the zero polynomial we set $d(0)=0$.

When $k=1$, we write $\mathcal P_d(\mathbb K^n)$ instead of
$\mathcal P_d(\mathbb K^n;\mathbb K)$. For nonzero
$f\in\mathcal P_d(\mathbb K^n)$, let
\[
[f]_2
:=
\Bigl(
\sum_{|\mathbf m|=d(f)}
|a_{\mathbf m}(f)|^2
\Bigr)^{1/2}
\]
and let $[0]_2=0$.

When $k=2$ and $d(f)\ge1$, let
$[[f]]_2$ denote the square root
of the determinant of the Gram matrix of the leading
monomial coefficients of $f_1$ and $f_2$, namely
\[
[[f]]_2^2:=
\det
\begin{pmatrix}
\sum\limits_{|\mathbf p|=d(f)}|a_{\mathbf p}(f_1)|^2 &
\sum\limits_{|\mathbf p|=d(f)}
a_{\mathbf p}(f_1)\overline{a_{\mathbf p}(f_2)}
\\
\sum\limits_{|\mathbf p|=d(f)}
\overline{a_{\mathbf p}(f_1)}a_{\mathbf p}(f_2) &
\sum\limits_{|\mathbf p|=d(f)}|a_{\mathbf p}(f_2)|^2
\end{pmatrix}.
\]
For constant $f$, we set $[[f]]_2=0$.

For fixed ${\bf r}\in{\bf N}^2_n:=
\{(r_1,r_2)\in\mathbb N^2\colon
1\le r_1<r_2\le n\}$, we define the partial Jacobian matrix by
\[
J_f^{\bf r}
:=
\begin{pmatrix}
\partial_{z_{r_1}}f_1 & \partial_{z_{r_2}}f_1
\\
\partial_{z_{r_1}}f_2 & \partial_{z_{r_2}}f_2
\end{pmatrix}.
\]

The key step in the proof of Theorem \ref{th-main}
is the following dimension-free algebraic estimate.
It improves the estimate obtained in
\cite{Kos26} by replacing the max-type coefficient quantities used there
with their natural $\ell^2$ counterparts.

\begin{theorem}\label{th-algebr}
Let $d, n\in\mathbb N$, $n\ge2$, and let
$f=(f_1,f_2)\in\mathcal P_d(\mathbb C^n;\mathbb C^2)$. Then
\[
\sum_{{\bf r}\in{\bf N}_n^2}[\det J_f^{\bf r}]_2^2
\ge
2^{-39d}[[f]]_2^2.
\]
\end{theorem}

We note that the exponential decay in $d$ of the constant in this estimate is unavoidable, see Example~\ref{ex-exponential-decay} below.

The coefficient Gram determinant $[[f]]_2^2$ is directly related to the
Gram determinant of the highest-order derivatives.
Indeed, for $f_1,f_2\in\mathcal P_d(\mathbb R^n)$, we have
\[
\langle D^df_1,D^df_2\rangle_{(\mathbb R^n)^{\otimes d}}
=
d!\sum_{|{\bf m}|=d}
m_1!\ldots m_n!\,
a_{\bf m}(f_1)a_{\bf m}(f_2).
\]
Comparing the Cauchy--Binet expansions of the corresponding Gram
determinants and using $1\le m_1!\ldots m_n!\le d!$, we obtain,
for $f=(f_1,f_2)\in\mathcal P_d(\mathbb R^n;\mathbb R^2)$
with $d(f)=d$,
\begin{equation}\label{eq-equiv}
(d!)^2[[f]]_2^2
\le
\det G_{(\mathbb R^n)^{\otimes d}}(D^df_1,D^df_2)
\le
(d!)^4[[f]]_2^2.
\end{equation}

We combine the algebraic estimate from Theorem~\ref{th-algebr} with the following
extension of \eqref{eq-MB-2} to bounded densities.

\begin{theorem}\label{th-MB}
Let $M>0$, $d\in\mathbb N$, and let
$\mu=\mu_1\otimes\ldots\otimes\mu_n$, where each $\mu_j$ is a
probability measure on $\mathbb R$ with a density bounded by $M$.
Then, for every $f\in\mathcal P_d(\mathbb R^n)$,
one has
\[
\int_{\mathbb R^n}|f|^2\,d\mu
\ge
2^{-16d(f)}M^{-2d(f)}[f]_2^2.
\]
\end{theorem}

Theorems~\ref{th-algebr} and~\ref{th-MB} together imply
Theorem~\ref{th-main}. Indeed, let $\mu$ be as in
Theorem~\ref{th-main}. By the Cauchy--Binet formula,
\[
\det G_{\mathbb R^n}(\nabla f_1,\nabla f_2)
=
\sum_{{\bf r}\in{\bf N}_n^2}|\det J_f^{\bf r}|^2.
\]
Each polynomial $\det J_f^{\bf r}$ has degree at most $2d-2\le2d$.
Hence, Theorem~\ref{th-MB} yields
\[
\int_{\mathbb R^n}|\det J_f^{\bf r}|^2\,d\mu
\ge
2^{-32d}[\det J_f^{\bf r}]_2^2.
\]
Summing these estimates and applying Theorem~\ref{th-algebr}, we obtain
\[
\int_{\mathbb R^n}
\det G_{\mathbb R^n}(\nabla f_1,\nabla f_2)\,d\mu
\ge
2^{-32d}
\sum_{{\bf r}\in{\bf N}_n^2}[\det J_f^{\bf r}]_2^2
\ge
2^{-71d}[[f]]_2^2.
\]
If $d(f)<d$, then $D^df_1=D^df_2=0$, and the assertion of
Theorem~\ref{th-main} is immediate. Otherwise, $d(f)=d$,
and \eqref{eq-equiv} completes the proof.

\subsection{Structure of the paper}

The rest of the paper is organized as follows. In
Section~\ref{sect-algebr}, we prove the key algebraic estimate,
Theorem~\ref{th-algebr}. In Section~\ref{sect-MB}, we prove
Theorem~\ref{th-MB}. Finally, in Section~\ref{sect-exampl}, we show that the corresponding
extensions of Theorem~\ref{th-main} to $\mathbb R^3$-valued polynomial
mappings and to non-product isotropic log-concave measures fail in general,
and that the exponential decay in the degree of the polynomial in
Theorem~\ref{th-algebr} is unavoidable.

\section{Algebraic estimate}
\label{sect-algebr}

\subsection{Lemmas on homogeneous polynomials}

\begin{lemma}\label{lem-hom-1}
Let $d\in \mathbb{N}$ and let
\[
U(t,s)
=
\sum_{m=0}^d\alpha_mt^ms^{d-m}
\quad\text{and}\quad
V(t,s)
=
\sum_{m=0}^d\beta_mt^ms^{d-m},
\]
with $\alpha_m,\beta_m\in\mathbb C$. 
Then
\[
[\partial_tU\partial_sV-\partial_tV\partial_sU]_2^2
\ge
2^{-30 d-6}\sum_{0\le r<m\le d}
|\alpha_m\beta_r-\alpha_r\beta_m|^2.
\]
In particular, 
\[
\bigl[\partial_tU\partial_sV-\partial_tV\partial_sU\bigr]_2^2
\ge
2^{-30 d-6}
\bigl|U(1, 0)V(0, 1) - U(0, 1)V(1, 0)\bigr|^2.
\]
\end{lemma}

\begin{proof}
If $d=1$, then
\[
\partial_tU\partial_sV-\partial_tV\partial_sU
=
\alpha_1\beta_0-\alpha_0\beta_1
\]
and the assertion holds with the constant $1\ge 2^{-36}$. We therefore assume that $d\ge2$.

Let
\[
\alpha=(\alpha_0,\ldots,\alpha_d)
\quad\text{and}\quad
\beta=(\beta_0,\ldots,\beta_d),
\]
and let
\[
\langle\alpha,\beta\rangle
=
\sum_{m=0}^d\alpha_m\overline{\beta_m}.
\]
We first consider the case
\begin{equation}\label{eq-assumpt}
|\alpha|=|\beta|=1,
\quad
\langle\alpha,\beta\rangle=0.
\end{equation}

Let
\[
u(z):=U(z,1)
\quad\text{and}\quad
v(z):=V(z,1).
\]
Then $v$ is a polynomial of degree $k\le d$
and it can be written as
\[
v(z)= A\prod_{j=1}^{k}(z- z_j).
\]
In particular,
\begin{equation}\label{eq-control}
1=|\beta|\le\sum_{m=0}^d|\beta_m|
\le |A|\sum_{S\subset \{1, \ldots, k\}}\prod_{j\in S}|z_j|= 
|A|\prod_{j=1}^{k}(1+|z_j|).
\end{equation}

We note that, for $t_0\in [0, 1]$,
\begin{align*}
\int_0^1\ln|t-t_0|\, dt 
&= \int_0^{t_0}\ln(t_0-t)\, dt + \int_{t_0}^1\ln(t-t_0)\, dt
\\
&= \int_0^{t_0}\ln t\, dt + \int_0^{1-t_0}\ln t\, dt
\ge 2\int_0^1\ln t\, dt = -2.
\end{align*}
If $t_0>1$, then
\[
\int_0^1\ln|t-t_0|\, dt
= \int_0^1\ln(t_0-t)\, dt\ge 
\int_0^1\ln(1-t)\, dt = -1\ge -2,
\]
and if $t_0<0$, similarly,
\[
\int_0^1\ln|t-t_0|\, dt
= \int_0^1\ln(t-t_0)\, dt\ge 
\int_0^1\ln t\, dt = -1\ge -2.
\]
Therefore, when $0\le |z_j|\le 3$, we have
\[
\int_1^2\ln\frac{|r - |z_j||}{1+|z_j|}\, dr
= \int_0^1\ln|t - (|z_j|-1)|\, dt - \ln(1+|z_j|)
\ge -4.
\]
When $|z_j|>3$ and $r\in [1, 2]$, 
\[
\frac{|r - |z_j||}{1+|z_j|} = 
\frac{|z_j| - r}{1+|z_j|} = 1 - \frac{r+1}{1+|z_j|}\ge \frac{1}{4}
\]
and 
\[
\int_1^2\ln\frac{|r - |z_j||}{1+|z_j|}\, dr\ge -\ln 4\ge -4.
\]
Thus, for every $j\in \{1, \ldots, k\}$, we have
\[
\int_1^2\ln\frac{|r - |z_j||}{1+|z_j|}\, dr\ge -4
\]
and therefore
\[
\int_1^2\ln\prod_{j=1}^k\frac{|r - |z_j||}{1+|z_j|}\, dr\ge -4k.
\]
Hence, there exists $r\in [1, 2]$, such that
\[
\prod_{j=1}^k\frac{|r - |z_j||}{1+|z_j|} \ge e^{-4k}\ge e^{-4d}.
\]
By \eqref{eq-control}, for $|z|=r$, we have
\begin{equation}\label{eq-lower}
|v(z)| \ge |A|
\prod_{j=1}^{k}|r - |z_j||\ge e^{-4d}.
\end{equation}
In particular, $v$ has no zeros on this circle.

Let
\[
J(t,s):=\partial_tU\partial_sV-\partial_tV\partial_sU.
\]
By homogeneity of $U$ and $V$,
\[
\partial_sU(z,1)=du(z)-zu'(z)
\quad\text{and}\quad
\partial_sV(z,1)=dv(z)-zv'(z).
\]
Hence,
\begin{align*}
J(z,1)=
(\partial_tU\partial_sV-\partial_tV\partial_sU)(z,1)
&=
u'(z)\bigl(dv(z)-zv'(z)\bigr)
-
v'(z)\bigl(du(z)-zu'(z)\bigr)
\\
&=
d\bigl(u'(z)v(z)-v'(z)u(z)\bigr)
\end{align*}
and therefore
\[
\Bigl(\frac{u}{v}\Bigr)'(z)=\frac{J(z,1)}{d(v(z))^2}
\]
wherever $v(z)\ne0$. 
Fix $z_*$ on the circle $|z|=r$
and set $c:=\frac{u(z_*)}{v(z_*)}$. Joining $z_*$ to any other point
of this circle by an arc of length at most $4\pi$, we obtain
\[
\sup_{|z|=r}|u(z)-cv(z)|
\le
\frac{4\pi}{d}\,
\frac{
\max\limits_{|z|=r}|v(z)|
\max\limits_{|z|=r}|J(z,1)|
}{
\bigl(\min\limits_{|z|=r}|v(z)|\bigr)^2
}.
\]

Since $r\le2$ and $|\beta|=1$, the Cauchy--Schwarz inequality gives
\[
\sup_{|z|=r}|v(z)|
\le
\biggl(\sum_{m=0}^d r^{2m}\biggr)^{1/2}
\le
\sqrt{d+1}\,2^d.
\]
Writing
\[
J(t,s)=\sum_{j=0}^{2d-2}B_jt^js^{2d-2-j},
\]
we similarly obtain
\[
\sup_{|z|=r}|J(z,1)|
\le
\biggl(\sum_{j=0}^{2d-2}|B_j|^2\biggr)^{1/2}
\biggl(\sum_{j=0}^{2d-2}r^{2j}\biggr)^{1/2}
\le
\sqrt{2d-1}\,2^{2d-2}[J]_2.
\]
Thus, using \eqref{eq-lower}, we obtain
\[
\sup_{|z|=r}|u(z)-cv(z)|
\le
\frac{\pi\sqrt{(d+1)(2d-1)}}{d}
(8e^8)^d[J]_2\le 8 (8e^8)^d[J]_2
\le 2^{15 d+3}[J]_2.
\]

On the other hand, the assumptions \eqref{eq-assumpt} imply
\[
\sum_{m=0}^d|\alpha_m-c\beta_m|^2
=|\alpha -c\beta|^2
=
1+|c|^2
\ge1.
\]
Since $r\ge1$, we have
\begin{align*}
1&\le
\sum_{m=0}^d|\alpha_m-c\beta_m|^2
\le
\sum_{m=0}^d|\alpha_m-c\beta_m|^2r^{2m}
\\
&=
\frac1{2\pi}\int_0^{2\pi}
|u(re^{i\theta})-cv(re^{i\theta})|^2\,d\theta
\le
\sup_{|z|=r}|u(z)-cv(z)|^2.
\end{align*}
Consequently,
\begin{equation}\label{eq-inter}
[\partial_tU\partial_sV-\partial_tV\partial_sU]_2^2
=[J]_2^2
\ge
2^{-30 d-6}
\end{equation}
under the assumption \eqref{eq-assumpt}.

For general coefficient vectors $\alpha$ and $\beta$,
by the Cauchy--Binet formula,
\[
\sum_{0\le r<m\le d}
|\alpha_m\beta_r-\alpha_r\beta_m|^2
=
|\alpha|^2|\beta|^2-|\langle\alpha,\beta\rangle|^2.
\]
If this quantity is zero, there is nothing to prove. Otherwise, set
\[
c_0:=\frac{\langle\beta,\alpha\rangle}{|\alpha|^2},
\quad
\beta^\perp:=\beta-c_0\alpha.
\]
Then
$\langle\alpha,\beta^\perp\rangle=0$
and
\[
\sum_{0\le r<m\le d}
|\alpha_m\beta_r-\alpha_r\beta_m|^2
=
|\alpha|^2|\beta^\perp+c_0\alpha|^2-|\langle\alpha,\beta^\perp + c_0\alpha\rangle|^2
=
|\alpha|^2|\beta^\perp|^2.
\]
Since 
\[
\partial_t\Bigl(\frac{U}{|\alpha|}\Bigr)\partial_s\Bigl(\frac{V-c_0U}{|\beta^\perp|}\Bigr)-\partial_t\Bigl(\frac{V-c_0U}{|\beta^\perp|}\Bigr)\partial_s\Bigl(\frac{U}{|\alpha|}\Bigr)
= \frac{\partial_tU\partial_sV-\partial_tV\partial_sU}{|\alpha||\beta^\perp|},
\]
the assertion follows from the 
already proved inequality \eqref{eq-inter} under the assumption \eqref{eq-assumpt}.
\end{proof}

\begin{lemma}\label{lem-hom-2}
Let $k\in \mathbb{Z}_+$ and let
\[
H_{pq}(z)
=
\sum_{|\mathbf m|=k} a_{\bf m}(H_{pq})z^{\mathbf m},\quad 1\le p,q\le n,
\]
be a collection of degree $k$ homogeneous polynomials in $n$ variables with complex coefficients.
Let
\[
e^{2\pi ix}:=(e^{2\pi ix_1}, \ldots, e^{2\pi ix_n})
\quad\text{and}\quad
e^{2\pi iy}:=(e^{2\pi iy_1}, \ldots, e^{2\pi iy_n})
\]
and let
\[
K_{x,y}(t,s):=\sum_{1\le p,q\le n}e^{2\pi i(x_p+y_q)}
H_{pq}(t e^{2\pi ix}+se^{2\pi iy}).
\]
Then
\[
\int_{[0,1]^{2n}} [K_{x,y}]_2^2\,dxdy
\le
2^{2k}(k+1)^3
\sum_{1\le p,q\le n}[H_{pq}]_2^2.
\]
\end{lemma}

\begin{proof}
We note that
\begin{align*}
H_{pq}(t e^{2\pi ix}+se^{2\pi iy})
&=
\sum_{|\mathbf m|=k} a_{\mathbf m}(H_{pq})
\prod_{j=1}^n(te^{2\pi ix_j}+se^{2\pi iy_j})^{m_j}
\\
&=
\sum_{|\mathbf m|=k} a_{\mathbf m}(H_{pq})
\sum_{r_1=0}^{m_1}\ldots
\sum_{r_n=0}^{m_n}
\prod_{j=1}^n\binom{m_j}{r_j}
t^{|\mathbf r|}s^{k-|\mathbf r|}
e^{2\pi i(\langle x,\mathbf r\rangle+
\langle y,\mathbf m-\mathbf r\rangle)}.
\end{align*}
Hence,
\begin{align*}
K_{x,y}(t,s)
&=
\sum_{1\le p,q\le n}
\sum_{|\mathbf m|=k} a_{\mathbf m}(H_{pq})
\sum_{r_1=0}^{m_1}\ldots
\sum_{r_n=0}^{m_n}
\prod_{j=1}^n\binom{m_j}{r_j}
t^{|\mathbf r|}s^{k-|\mathbf r|}
e^{2\pi i(\langle x,\mathbf r+e_p\rangle+
\langle y,\mathbf m-\mathbf r+e_q\rangle)}
\\
&=
\sum_{\ell=0}^k
\biggl(
\sum_{1\le p,q\le n}
\sum_{|\mathbf m|=k}
\sum_{\substack{|\mathbf r|=\ell\\0\le r_j\le m_j}}
a_{\mathbf m}(H_{pq})
\prod_{j=1}^n\binom{m_j}{r_j}
e^{2\pi i(\langle x,\mathbf r+e_p\rangle+
\langle y,\mathbf m-\mathbf r+e_q\rangle)}
\biggr)
t^\ell s^{k-\ell}.
\end{align*}
Consider the trigonometric polynomial
\[
K_\ell(x,y):=
\sum_{1\le p,q\le n}
\sum_{|\mathbf m|=k}
\sum_{\substack{|\mathbf r|=\ell\\0\le r_j\le m_j}}
a_{\mathbf m}(H_{pq})
\prod_{j=1}^n\binom{m_j}{r_j}
e^{2\pi i(\langle x,\mathbf r+e_p\rangle+
\langle y,\mathbf m-\mathbf r+e_q\rangle)}.
\]
We note that distinct functions
\[
e^{2\pi i(\langle x,\mathbf u\rangle+
\langle y,\mathbf v\rangle)}
\]
are orthonormal in $L^2([0, 1]^{2n})$.
We now estimate how many times each element of this orthonormal family appears
in the sum $K_\ell$. To do this, for fixed
$\mathbf u,\mathbf v\in\mathbb Z_+^n$, let
$\mathbf m,\mathbf r,p,q$ be such that
\[
\mathbf u=\mathbf r+e_p,
\quad
\mathbf v=\mathbf m-\mathbf r+e_q,
\]
where $1\le p,q\le n$, $|\mathbf m|=k$, $|\mathbf r|=\ell$, and
$r_j\le m_j$ for every $j\in\{1,\ldots,n\}$.
In this case, necessarily
\[
p\in\operatorname{supp}\mathbf u
\quad\text{and}\quad
q\in\operatorname{supp}\mathbf v,
\]
and $\mathbf m$ and $\mathbf r$ are uniquely defined once $p$ and $q$
are chosen, since
\[
\mathbf r=\mathbf u-e_p
\quad\text{and}\quad
\mathbf m=\mathbf v+\mathbf r-e_q
=\mathbf v+\mathbf u-e_p-e_q.
\]
Therefore, the element
$e^{2\pi i(\langle x,\mathbf u\rangle+
\langle y,\mathbf v\rangle)}$
appears in the sum for $K_\ell$ no more than
$|\operatorname{supp}\mathbf u|\cdot|\operatorname{supp}\mathbf v|$
times.
We have
\[
|\operatorname{supp}\mathbf u|
\le |\mathbf u|
=
|\mathbf r|+1
\le k+1
\]
and
\[
|\operatorname{supp}\mathbf v|
\le |\mathbf v|
=
|\mathbf m|-|\mathbf r|+1
\le k+1.
\]
Hence,
\[
|\operatorname{supp}\mathbf u|\cdot|\operatorname{supp}\mathbf v|
\le (k+1)^2.
\]
Thus, each element
$e^{2\pi i(\langle x,\mathbf u\rangle+
\langle y,\mathbf v\rangle)}$
of the orthonormal family appears in the sum for $K_\ell$ no more than
$(k+1)^2$ times, and therefore, by the Cauchy--Schwarz inequality,
\[
\int_{[0,1]^{2n}}|K_\ell(x,y)|^2\,dxdy
\le
(k+1)^2
\sum_{1\le p,q\le n}
\sum_{|\mathbf m|=k}
|a_{\mathbf m}(H_{pq})|^2
\sum_{\substack{|\mathbf r|=\ell\\0\le r_j\le m_j}}
\prod_{j=1}^n\binom{m_j}{r_j}^2.
\]
For the last sum we have
\[
\sum_{\substack{|\mathbf r|=\ell\\0\le r_j\le m_j}}
\prod_{j=1}^n\binom{m_j}{r_j}^2
\le
\biggl(
\sum_{r_1=0}^{m_1}\ldots
\sum_{r_n=0}^{m_n}
\prod_{j=1}^n\binom{m_j}{r_j}
\biggr)^2
=
2^{2|\mathbf m|}
=
2^{2k}.
\]
Thus,
\[
\int_{[0,1]^{2n}}[K_{x,y}]_2^2\,dxdy
=
\sum_{\ell=0}^k
\int_{[0,1]^{2n}}|K_\ell(x,y)|^2\,dxdy
\le
2^{2k}(k+1)^3
\sum_{1\le p,q\le n}[H_{pq}]_2^2,
\]
which gives the announced estimate.
\end{proof}

\begin{lemma}\label{lem-hom-3}
Let $d\in \mathbb{N}$ and
let $f=(f_1,f_2)\in \mathcal{P}_d(\mathbb{C}^n;\mathbb{C}^2)$,
\[
f_1(z)=\sum_{|\mathbf m|=d} a_{\bf m}(f_1)z^{\mathbf m}
\quad\text{and}\quad
f_2(z)=\sum_{|\mathbf m|=d} a_{\bf m}(f_2)z^{\mathbf m}.
\]
Then
\[
\sum_{{\bf r}\in{\bf N}_n^2}[\det J_f^{\bf r}]_2^2
\ge
2^{-39 d}[[f]]_2^2.
\]
\end{lemma}

\begin{proof}
Let 
\[
U_{x, y}(t, s):= f_1(t e^{2\pi ix}+se^{2\pi iy})
\quad\text{and}\quad
V_{x, y}(t, s):= f_2(t e^{2\pi ix}+se^{2\pi iy}).
\]
By Lemma \ref{lem-hom-1}, for every fixed $x, y\in[0, 1]^n$,
\[
\bigl[\partial_t U_{x, y}\partial_s V_{x, y} - \partial_t V_{x, y}\partial_s U_{x, y}\bigr]_2^2
\ge 2^{-30 d-6}\bigl|f_1(e^{2\pi ix})f_2(e^{2\pi iy}) - f_2(e^{2\pi ix})f_1(e^{2\pi iy})\bigr|^2.
\]
We note that
\[
f_1(e^{2\pi ix})f_2(e^{2\pi iy}) - f_2(e^{2\pi ix})f_1(e^{2\pi iy})
=
\sum_{|\mathbf p|= d}\sum_{|\mathbf q|= d} 
\bigl(a_{\bf p}(f_1)a_{\bf q}(f_2) - a_{\bf p}(f_2)a_{\bf q}(f_1)\bigr)e^{2\pi i(\langle x, \mathbf p \rangle + \langle y, \mathbf q\rangle)}.
\]
Hence,
\begin{align*}
\int_{[0, 1]^{2n}}\bigl[\partial_t U_{x, y}\partial_s V_{x, y} - \partial_t V_{x, y}\partial_s U_{x, y}\bigr]_2^2\, dxdy
&\ge 2^{-30 d-6}\sum_{|\mathbf p| = |\mathbf q|= d}
\bigl|a_{\bf p}(f_1)a_{\bf q}(f_2) - a_{\bf p}(f_2)a_{\bf q}(f_1)\bigr|^2
\\
&=2^{-30 d-5}[[f]]_2^2,
\end{align*}
where the last equality 
follows from the Cauchy--Binet formula.

On the other hand,
\[
\partial_t U_{x, y}\partial_s V_{x, y} - \partial_t V_{x, y}\partial_s U_{x, y}=
\sum_{p=1}^n\sum_{q=1}^n e^{2\pi i (x_p+y_q)}\bigl(\partial_{z_p}f_1\partial_{z_q}f_2-\partial_{z_p}f_2\partial_{z_q}f_1\bigr)(t e^{2\pi ix}+se^{2\pi iy}).
\]
Each polynomial
\[
H_{pq}:=
\partial_{z_p}f_1\partial_{z_q}f_2
-
\partial_{z_p}f_2\partial_{z_q}f_1
\]
is homogeneous of degree $k=2d-2$.
By Lemma \ref{lem-hom-2},
\begin{align*}
&\int_{[0,1]^{2n}}
\bigl[\partial_t U_{x,y}\partial_s V_{x,y}
-
\partial_t V_{x,y}\partial_s U_{x,y}\bigr]_2^2\,dxdy
\le
2^{4d-4}(2d-1)^3
\sum_{1\le p,q\le n}[H_{pq}]_2^2
\\
&=
2^{4d-3}(2d-1)^3
\sum_{1\le p<q\le n}[H_{pq}]_2^2
\le 2^{4d}d^3\sum_{1\le p<q\le n}[H_{pq}]_2^2.
\end{align*}
Combining the upper and lower estimates and using
$(2d)^3\le 2^{3d}$ yields the claimed bound.
\end{proof}

\subsection{Proof of Theorem \ref{th-algebr}}

If $[[f]]_2=0$, there is nothing to prove. Assume that $[[f]]_2>0$
and set
$m:=d(f)\ge1$.
Let
\[
f_1=\sum_{j=0}^{m} f_1^{(j)}
\quad\text{and}\quad
f_2=\sum_{j=0}^{m} f_2^{(j)}
\]
be the decompositions of $f_1$ and $f_2$ into homogeneous components, where
$f_1^{(j)}$ and $f_2^{(j)}$ are homogeneous of degree $j$. Then
\[
\det J_f^{\bf r}
=
\sum_{s=1}^{m}\sum_{j=1}^{m}
\Bigl(
\partial_{z_{r_1}}f_1^{(s)}\,\partial_{z_{r_2}}f_2^{(j)}
-
\partial_{z_{r_2}}f_1^{(s)}\,\partial_{z_{r_1}}f_2^{(j)}
\Bigr).
\]
The summand with indices $(s,j)$ is homogeneous of degree $s+j-2$. Hence the
homogeneous component of $\det J_f^{\bf r}$ of degree $2m-2$ is
$\det J_{f^{(m)}}^{\bf r}$,
where
$f^{(m)}:=\bigl(f_1^{(m)},f_2^{(m)}\bigr)$.
Therefore,
\[
[\det J_f^{\bf r}]_2^2
\ge
\bigl[\det J_{f^{(m)}}^{\bf r}\bigr]_2^2,
\]
and, by Lemma \ref{lem-hom-3},
\[
\sum_{{\bf r}\in{\bf N}_n^2}[\det J_f^{\bf r}]_2^2
\ge
\sum_{{\bf r}\in{\bf N}_n^2}
\bigl[\det J_{f^{(m)}}^{\bf r}\bigr]_2^2
\ge
2^{-39m}[[f^{(m)}]]_2^2
=
2^{-39m}[[f]]_2^2
\ge
2^{-39d}[[f]]_2^2,
\]
which completes the proof.
\qed

\section{Proof of Theorem \ref{th-MB}}
\label{sect-MB}

If $d(f)=0$, the assertion is immediate. Assume therefore that $d(f)\ge1$.
	
Recall that for $u\in \mathcal{P}_k(\mathbb{R})$ with $k=d(u)\ge 1$ and $\varepsilon>0$, 
by \cite[Corollary~3.6]{Kos25},
one has
\[
\mu_j\bigl(t\in \mathbb{R}\colon |u(t)|\le \varepsilon\bigr)\le \frac{8ekM}{|a_k(u)|^{1/k} (k!)^{1/k}}\, \varepsilon^{1/k}
\le 8e^2M|a_k(u)|^{-1/k}\varepsilon^{1/k}.
\]
In particular,
\[
\int_{\mathbb{R}}|u|^2\, d\mu_j
\ge \varepsilon^2 \bigl(1-\mu_j(|u|\le \varepsilon)\bigr),
\]
and taking $\varepsilon = (16 e^2M)^{-k}|a_k(u)|$,
we obtain
\begin{equation}\label{eq-1-dim}
\int_{\mathbb{R}}|u|^2\, d\mu_j
\ge 2^{-1}(16 e^2M)^{-2k} |a_k(u)|^2.
\end{equation}

First, assume that all the measures $\mu_j$ are compactly
supported. Let $u_{j,0}=1$, $u_{j,1}, \ldots, u_{j, d}$ be the orthonormal polynomials of the corresponding degrees with respect to the measure $\mu_j$, and let
\[
U_{\bf m}(x):=\prod_{j=1}^n u_{j,m_j}(x_j).
\]
Then the family $\{U_{\bf m}\colon |{\bf m}|\le d\}$ forms an orthonormal basis in $\mathcal{P}_d(\mathbb{R}^n)$ with respect to the inner product in $L^2(\mu)$.
For ${\bf m}\in \mathbb Z_+^n$ with $|{\bf m}|\le d$, applying \eqref{eq-1-dim} for each $j$ with $m_j>0$, we obtain
\[
\prod_{j=1}^n \bigl|a_{m_j}(u_{j, m_j})\bigr|^{-2}
\ge 2^{-|{\bf m}|}(16e^2M)^{-2|{\bf m}|} = (16\sqrt{2}e^2M)^{-2|{\bf m}|}\ge 2^{-16|{\bf m}|}M^{-2|{\bf m}|}.
\]
Moreover, for $f\in \mathcal{P}_d(\mathbb{R}^n)$
and for ${\bf m}\in\mathbb Z_+^n$ with $|{\bf m}|=d(f)$, we have
\[
\langle f, U_{\bf m}(x)\rangle_{L^2(\mu)}
= a_{\bf m}(f)\langle x^{\bf m}, U_{\bf m}(x)\rangle_{L^2(\mu)}
= a_{\bf m}(f)\prod_{j=1}^n \bigl(a_{m_j}(u_{j, m_j})\bigr)^{-1}.
\]
Therefore, for $f\in \mathcal{P}_d(\mathbb{R}^n)$,
\[
\int_{\mathbb{R}^n} |f|^2\, d\mu
\ge \sum_{|{\bf m}|=d(f)}|\langle f, U_{\bf m}(x)\rangle_{L^2(\mu)}|^2
\ge 2^{-16d(f)}M^{-2d(f)}[f]_2^2.
\]

For general measures, we choose $R$ sufficiently large that
$\mu_j([-R,R])>0$ for every $j$, and define
\[
\mu_{j,R}(A):=
\frac{\mu_j(A\cap[-R,R])}{\mu_j([-R,R])},
\quad
\mu_R:=\mu_{1,R}\otimes\ldots\otimes\mu_{n,R}.
\]
The densities of these measures are bounded by
\[
M_R:=
\frac{M}{\min\limits_{1\le j\le n}\mu_j([-R,R])},
\]
and $M_R\to M$ as $R\to\infty$. Applying the
estimate from the compactly supported
case gives
\[
\int_{[-R,R]^n}|f|^2\,d\mu
\ge
\biggl(\prod_{j=1}^n\mu_j([-R,R])\biggr)
2^{-16d(f)}M_R^{-2d(f)}[f]_2^2.
\]
Letting $R\to\infty$ and using monotone convergence proves the
assertion, including when the integral is infinite.
\qed

\section{Examples}
\label{sect-exampl}

\begin{example}\label{ex-three-components}
Let
\[
f=(f_1,f_2,f_3)\colon\mathbb R^3\to\mathbb R^3,
\quad
f_1(x)=x_1^2-1,
\quad
f_2(x)=x_1x_2,
\quad
f_3(x)=x_2^2-1.
\]
Then
\[
\nabla f_1=2x_1e_1,
\quad
\nabla f_2=x_2e_1+x_1e_2,
\quad
\nabla f_3=2x_2e_2.
\]
All three gradients belong to $\operatorname{span}\{e_1,e_2\}$.
Consequently,
\[
\det G_{\mathbb R^3}
(\nabla f_1,\nabla f_2,\nabla f_3)=0
\]
at every point of $\mathbb R^3$.
On the other hand,
\[
D^2f_1=2e_1\otimes e_1,
\quad
D^2f_2=e_1\otimes e_2+e_2\otimes e_1,
\quad
D^2f_3=2e_2\otimes e_2.
\]
These tensors are mutually orthogonal, and their squared norms are
$4$, $2$, and $4$, respectively. Hence,
\[
\det G_{(\mathbb R^3)^{\otimes2}}
(D^2f_1,D^2f_2,D^2f_3)
=
\det
\begin{pmatrix}
4&0&0\\
0&2&0\\
0&0&4
\end{pmatrix}
=
32.
\]
Therefore, there is no positive constant $C$ such that
\[
\int_{\mathbb R^3}
\det G_{\mathbb R^3}
(\nabla f_1,\nabla f_2,\nabla f_3)\,d\gamma_3
\ge
C
\det G_{(\mathbb R^3)^{\otimes2}}
(D^2f_1,D^2f_2,D^2f_3)
\]
for every quadratic mapping $f\colon\mathbb R^3\to\mathbb R^3$.
\end{example}

\begin{example}\label{ex-nonproduct}
For $n\ge2$, let $\mu_n$ be the normalized Lebesgue measure on
\[
B_n:=\{x\in\mathbb R^n\colon |x|\le\sqrt{n+2}\}.
\]
Then $\mu_n$ is an isotropic log-concave measure on $\mathbb R^n$.
Let
\[
g(x):=|x|^2-n,
\quad
f_1(x):=x_1g(x),
\quad
f_2(x):=x_2g(x).
\]
We have
\[
\nabla f_1=
(g(x)+2x_1^2,2x_1x_2,\ldots,2x_1x_n)
\quad
\text{and}
\quad
\nabla f_2=
(2x_2x_1,g(x)+2x_2^2,2x_2x_3,\ldots,2x_2x_n).
\]
Hence,
\[
\langle\nabla f_r,\nabla f_j\rangle
=
g^2\delta_{rj}+4(2|x|^2-n)x_rx_j,
\]
and therefore
\[
\det G_{\mathbb R^n}(\nabla f_1,\nabla f_2)
=
g^4+4g^2(2|x|^2-n)(x_1^2+x_2^2).
\]
By symmetry,
\[
\int_{\mathbb R^n}
\det G_{\mathbb R^n}(\nabla f_1,\nabla f_2)\,d\mu_n
=
\int_{\mathbb R^n}g^4\,d\mu_n
+
\frac8n
\int_{\mathbb R^n}
g^2|x|^2(2|x|^2-n)\,d\mu_n.
\]
It is readily verified that
\[
\int_{\mathbb R^n}g^2\,d\mu_n
=
\int_{\mathbb R^n}|x|^4\,d\mu_n-n^2
=
\frac{(n+2)^2n}{n+4}-n^2
=
\frac{4n}{n+4}.
\]
By the equivalence of moments of polynomials of a fixed degree under
log-concave measures~\cite{CW},
\[
\int_{\mathbb R^n}g^4\,d\mu_n
\le
C_1
\biggl(
\int_{\mathbb R^n}g^2\,d\mu_n
\biggr)^2
\le
C_2,
\]
where $C_1,C_2>0$ are absolute constants.
Since $0\le |x|^2\le n+2$ on $B_n$,
\[
|x|^2(2|x|^2-n)
\le
(n+2)(n+4).
\]
Consequently,
\[
\int_{\mathbb R^n}
\det G_{\mathbb R^n}(\nabla f_1,\nabla f_2)\,d\mu_n
\le
C_2+
\frac8n(n+2)(n+4)\frac{4n}{n+4}
=
C_2+32(n+2).
\]

On the other hand, the leading homogeneous parts of $f_1$ and $f_2$ are
\[
f_1^{(3)}(x)=x_1\sum_{j=1}^n x_j^2
\quad\text{and}\quad
f_2^{(3)}(x)=x_2\sum_{j=1}^n x_j^2,
\]
and therefore
$[[f]]_2^2=n^2$.
By \eqref{eq-equiv},
\[
\det
G_{(\mathbb R^n)^{\otimes3}}
(D^3f_1,D^3f_2)
\ge
36[[f]]_2^2
=
36n^2.
\]
Hence,
\[
\frac{
\int_{\mathbb R^n}
\det G_{\mathbb R^n}(\nabla f_1,\nabla f_2)\,d\mu_n
}{
\det
G_{(\mathbb R^n)^{\otimes3}}
(D^3f_1,D^3f_2)
}
\le
\frac{C_2+32(n+2)}{36n^2}
\longrightarrow0
\quad\text{as }n\to\infty.
\]
Thus, there is no dimension-free analogue of
Theorem~\ref{th-main} for isotropic log-concave measures outside the product case.
\end{example}

\begin{example}\label{ex-exponential-decay}
For odd $d\in\mathbb N$, consider $f=(f_1,f_2)$, where
\[
f_1(z_1,z_2)=(z_1+z_2)^d,
\quad
f_2(z_1,z_2)=(z_1-z_2)^d.
\]
Their coefficient vectors are orthogonal, and binomial identities give
$[[f]]_2^2=\binom{2d}{d}^2$.
Moreover,
\[
\det J_f^{(1,2)}
=
-2d^2(z_1^2-z_2^2)^{d-1},
\quad
[\det J_f^{(1,2)}]_2^2
=
4d^4\binom{2d-2}{d-1}.
\]
Thus, by Stirling's formula,
\[
\frac{[\det J_f^{(1,2)}]_2^2}{[[f]]_2^2}
=
\frac{4d^4\binom{2d-2}{d-1}}{\binom{2d}{d}^2}
\sim
\sqrt{\pi}\,d^{9/2}4^{-d}
\]
as $d\to\infty$ through odd integers.
Hence exponential decay in $d$ of the best constant in
Theorem~\ref{th-algebr} is unavoidable.
\end{example}

\section*{Use of AI Tools}

The present paper is a follow-up to \cite{Kos26} and grew out of the author's discussions with ChatGPT concerning possible improvements of the algebraic estimate introduced in \cite[Theorem~1.3]{Kos26}. The idea, suggested by ChatGPT, of reducing the problem to the two-dimensional case using trigonometric vectors made it possible to extend the estimate from an $\ell^\infty$-type result to an $\ell^2$-type result. The resulting estimate, proved here in Theorem~\ref{th-algebr}, provides the missing ingredient for answering the question of Nourdin, Nualart, and Poly in the Gaussian setting.

The author acknowledges the use of ChatGPT during the preparation of the manuscript to explore possible mathematical approaches and for language and stylistic editing. All AI-generated material used was reviewed and revised by the author. The author takes full responsibility for the content and final version of the paper.

\section*{Acknowledgements}

The author was supported by the AEI grants 
RYC2023-043616-I and
PID2025-169712NA-I00 funded by MICIU/AEI/10.13039/501100011033,
and by the Spanish State Research Agency, through the Severo Ochoa and Mar\'ia de Maeztu Program for Centers and
Units of Excellence in R\&D (CEX2020-001084-M).
The author thanks CERCA Programme (Generalitat de Catalunya) for institutional support.
{\sloppy
	
}

\end{document}